\documentclass[12pt,letterpaper,reqno]{amsart}

\usepackage{times}

\usepackage[T1]{fontenc}

\usepackage{mathrsfs}

\usepackage{latexsym}

\usepackage[dvips]{graphics}

\usepackage{epsfig}

\usepackage{amsmath,amsfonts,amsthm,amssymb,amscd}

\usepackage{verbatim}

\input amssym.def

\input amssym.tex

\newtheorem{thm}{Theorem}[section]

\newtheorem{lem}[thm]{Lemma}

\newtheorem{defi}[thm]{Definition}

\newtheorem{conj}[thm]{Conjecture}

\newtheorem{rek}[thm]{Remark}

\newcommand\bi{\begin{itemize}}

\newcommand\ei{\end{itemize}}

\newcommand\ben{\begin{enumerate}}

\newcommand\een{\end{enumerate}}

\newcommand{\N}{\mathbb{N}}

\newcommand{\Z}{\mathbb{Z}}

\newcommand{\gep}{\epsilon}  

\newcommand\be{\begin{equation}}

\newcommand\ee{\end{equation}}

\newcommand\bea{\begin{eqnarray}}

\newcommand\eea{\end{eqnarray}}

\newcommand{\foh}{\frac{1}{2}}  

\numberwithin{equation}{section}

\begin{document}

\title{Explicit constructions of infinite families of MSTD sets}

\author{Steven J. Miller}\email{Steven.J.Miller@williams.edu}

\address{Department of Mathematics and Statistics, Williams College, Williamstown, MA 01267}

\author{Brooke Orosz}\email{BOrosz@gc.cuny.edu}

\address{Department of Mathematics, The Graduate Center/CUNY, New York, NY 10016}

\author{Daniel Scheinerman}\email{Daniel\underline{\ }Scheinerman@brown.edu}

\address{Department of Mathematics, Brown University, Providence, RI 02912}

\subjclass[2000]{11P99 (primary).} \keywords{Sum
dominated sets, infinite families of MSTDs}

\date{\today}

\thanks{We thank Dan Katz for comments on an earlier draft. The first named author was partly supported by NSF grant DMS0600848. The second named author was supported by the George I. Alden UTRA at Brown University.
}

\begin{abstract} We explicitly construct infinite families of MSTD (more sums than differences) sets, i.e., sets where $|A+A| > |A-A|$. There are enough of these sets to prove that there exists a constant $C$ such that at least $C / r^4$ of the $2^r$ subsets of $\{1,\dots,r\}$ are MSTD sets; thus our family is significantly denser than previous constructions (whose densities are at most $f(r)/2^{r/2}$ for some polynomial $f(r)$). We conclude by generalizing our method to compare linear forms $\gep_1 A + \cdots + \gep_n A$ with $\gep_i \in \{-1,1\}$.
\end{abstract}

\maketitle


\section{Introduction}

Given a finite set of integers $A$, we define its sumset $A+A$ and difference set $A-A$ by \bea A + A & \ = \ & \{a_i + a_j: a_i, a_j \in A\} \nonumber\\ A - A & = & \{a_i - a_j: a_i, a_j \in A\}, \eea and let $|X|$ denote the cardinality of $X$. If $|A+A| > |A-A|$, then, following Nathanson, we call $A$ an MSTD (more sums than differences) set. As addition is commutative while subtraction is not, we expect that for a `generic' set $A$ we have $|A-A| > |A+A|$, as a typical pair $(x,y)$ contributes one sum and two differences; thus we expect MSTD sets to be rare.

Martin and O'Bryant \cite{MO} proved that, in some sense, this intuition is wrong. They considered the uniform model\footnote{This means each of the $2^n$ subsets of $\{1,\dots,n\}$ are equally likely to be chosen, or, equivalently, that the probability any $k \in \{1,\dots,n\}$ is in $A$ is just $1/2$.\label{footnote:unifmodel}} for choosing a subset $A$ of $\{1,\dots,n\}$, and showed that there is a positive probability that a random subset $A$ is an MSTD set (though, not surprisingly, the probability is quite small). However, the answer is very different for other ways of choosing subsets randomly, and if we decrease slightly the probability an element is chosen then our intuition is correct. Specifically, consider the binomial model with parameter $p(n)$, with $\lim_{n\to\infty} p(n) = 0$ and $n^{-1} = o(p(n))$ (so $p(n)$ doesn't tend to zero so rapidly that the sets are too sparse).\footnote{This model means that the probability $k \in \{1,\dots,n\}$ is in $A$ is $p(n)$.} Hegarty and Miller \cite{HM} recently proved that, in the limit as $n \to 0$, the percentage of subsets of $\{1,\dots,n\}$ that are MSTD sets tends to zero in this model.

Though MSTD sets are rare, they do exist (and, in the uniform model, are somewhat abundant by the work of Martin and O'Bryant). Examples go back to the 1960s. Conway is said to have discovered $\{0, 2, 3, 4, 7, 11, 12, 14\}$, while Marica \cite{Ma} gave $\{0, 1, 2, 4, 7, 8, 12, 14, 15\}$ in 1969 and Freiman and Pigarev \cite{FP} found $\{0, 1, 2, 4, 5$, $9, 12, 13$, $14, 16, 17$, $21, 24, 25, 26, 28, 29\}$ in 1973. Recent work includes infinite families constructed by Hegarty \cite{He} and Nathanson \cite{Na2}, as well as existence proofs by Ruzsa \cite{Ru1, Ru2, Ru3}.

Most of the previous constructions\footnote{An alternate method constructs an infinite family from a given MSTD set $A$ by considering $A_t = \{\sum_{i=1}^t a_i m^{i-1}: a_i \in A\}$. For $m$ sufficiently large, these will be MSTD sets; this is called the base expansion method. Note, however, that these will be very sparse. See \cite{He} for more details.} of infinite families of MSTD sets start with a symmetric set which is then `perturbed' slightly through the careful addition of a few elements that increase the number of sums more than the number of differences; see \cite{He,Na2} for a description of some previous constructions and methods. In many cases, these symmetric sets are arithmetic progressions; such sets are natural starting points because if $A$ is an arithmetic progression, then $|A+A| = |A-A|$.\footnote{As $|A+A|$ and $|A-A|$ are not changed by mapping each $x \in A$ to $\alpha x + \beta$ for any fixed $\alpha$ and $\beta$, we may assume our arithmetic progression is just $\{0,\dots,n\}$, and thus the cardinality of each set is $2n+1$.}

In this work we present a new method which takes an MSTD set satisfying certain conditions and constructs an infinite family of MSTD sets. While these families are not dense enough to prove a positive percentage of subsets of $\{1, \dots, r\}$ are MSTD sets, we are able to elementarily show that the percentage is at least $C / r^4$ for some constant $C$. Thus our families are far denser than those in \cite{He,Na2}; trivial counting\footnote{For example, consider the following construction of MSTD sets from \cite{Na2}: let $m, d, k \in \N$ with $m \ge 4$, $1 \le d \le m-1$, $d \neq m/2$, $k \ge 3$ if $d < m/2$ else $k \ge 4$. Set $B = [0,m-1] \backslash \{d\}$, $L = \{m-d, 2m-d, \dots , km-d\}$, $a^\ast = (k + 1)m-2d$ and $A = B \cup L \cup (a^\ast - B) \cup \{m\}$.
Then $A$ is an MSTD set. The width of such a set is of the order $km$. Thus, if we look at all triples $(m,d,k)$ with $km \le r$ satisfying the above conditions, these generate on the order of at most $\sum_{k \le r} \sum_{m \le r/k} \sum_{d \le m} 1 \ll r^2$, and there are of the order $2^r$ possible subsets of $\{0,\dots,r\}$; thus this construction generates a negligible number of MSTD sets. Though we write $f(r)/2^{r/2}$ to bound the percentage from other methods, a more careful analysis shows it is significantly less; we prefer this easier bound as it is already significantly less than our method. See for example Theorem 2 of \cite{He} for a denser example.} shows all of their infinite families give at most $f(r)2^{r/2}$ of the subsets of $\{1,\dots, r\}$ (for some polynomial $f(r)$) are MSTD sets, implying a percentage of at most $f(r) / 2^{r/2}$.\\

We first introduce some notation. The first is a common convention, while the second codifies a property which we've found facilitates the construction of MSTD sets. \\

\bi

\item We let $[a,b]$ denote all integers from $a$ to $b$; thus $[a,b] = \{n \in \Z: a \le n \le b\}$.\\

\item We say a set of integers $A$ has the property $P_n$ (or is a $P_n$-set) if both its sumset and its difference set contain all but the first and last $n$ possible elements (and of course it may or may not contain some of these fringe elements).\footnote{It is not hard to show that for fixed $0<\alpha\le1$ a random set drawn from $[1,n]$ in the uniform model is a $P_{\lfloor \alpha n\rfloor}$-set with probability approaching $1$ as $n\to\infty$.\label{footnote:beingpn}} Explicitly, let $a=\min{A}$ and $b=\max{A}$. Then $A$ is a $P_n$-set if \bea\label{eq:beingPnsetsum} [2a+n,\ 2b-n]  \ \subset\  A+A \eea and \bea\label{eq:beingPnsetdiff} [-(b-a)+n,\ (b-a)-n]\ \subset\ A-A.\eea \ \\ \

\ei

We can now state our construction and main result.

\begin{thm}\label{thm:mainconstruction}
  Let  $A=L\cup R$ be a $P_n$, MSTD set where  $L\subset[1,n]$, $R\subset[n+1,2n]$, and $1,2n\in A$;\footnote{Requiring $1, 2n \in A$ is quite mild; we do this so that we know the first and last elements of $A$.} see Remark \ref{rek:thmisnontrivial} for an example of such an $A$. Fix a $k \ge n$ and let $m$ be arbitrary. Let $M$ be any subset of $[n+k+1, n+k+m]$ with the property that it does not have a run of more than $k$ missing elements (i.e., for all $\ell \in [n+k+1,n+m+1]$ there is a $j \in [\ell, \ell+k-1]$ such that $j\in M$). Assume further that $n+k+1 \not\in M$ and set $A(M;k)=L \cup O_1 \cup M \cup O_2 \cup R'$, where $O_1=[n+1,n+k]$, $O_2=[n+k+m+1,n+2k+m]$ (thus the $O_i$'s are just sets of $k$ consecutive integers), and $R'=R+2k+m$. Then

  \ben \item $A(M;k)$ is an MSTD set, and thus we obtain an infinite family of distinct MSTD sets as $M$ varies;

  \item there is a constant $C > 0$ such that as $r\to\infty$ the percentage of subsets of $\{1,\dots,r\}$ that are in this family (and thus are MSTD sets) is at least $C / r^4$.

  \een
\end{thm}

\begin{rek}\label{rek:thmisnontrivial} In order to show that our theorem is not trivial, we must of course exhibit at least one $P_n$, MSTD set $A$ satisfying all our requirements (else our family is empty!). We may take the set\footnote{This $A$ is trivially modified from \cite{Ma} by adding 1 to each element, as we start our sets with 1 while other authors start with 0. We chose this set as our example as it has several additional nice properties that were needed in earlier versions of our construction which required us to assume slightly more about $A$.} $A = \{1, 2, 3, 5, 8, 9, 13, 15, 16\}$; it is an MSTD set as \bea A+A & \ = \ & \{2,3,4,5,6,7,8,9,10,11,12,13,14,15,16,17,18,19,20,21,\nonumber\\ & & \ \ \ \ \ 22,23,24,25,26,28,29,30,31,32\} \nonumber\\ A-A & \ = \ & \{-15,-14,-13,-12,-11,-10,-8,-7,-6,-5,-4,-3,-2,-1,\nonumber\\ & & \ \ \ \ \ 0,1,2,3,4,5,6,7,8,10,11,12,13,14,15\} \eea (so $|A+A| = 30 >29 = |A-A|$). $A$ is also a $P_n$-set, as \eqref{eq:beingPnsetsum} is satisfied since $[10,24] \subset A+A$ and \eqref{eq:beingPnsetdiff} is satisfied since $[-7,7] \subset A-A$.

For the uniform model, a subset of $[1,2n]$ is a $P_n$-set with high probability as $n\to\infty$, and thus examples of this nature are plentiful. For example, of the $1748$ MSTD sets with minimum $1$ and maximum $24$, $1008$ are $P_n$-sets.
\end{rek}

Unlike other estimates on the percentage of MSTD sets, our arguments are not probabilistic, and rely on explicitly constructing large families of MSTD sets. Our arguments share some similarities with the methods in \cite{He} (see for example Case I of Theorem 8) and \cite{MO}. There the fringe elements of the set were also chosen first. A random set was then added in the middle, and the authors argued that with high probability the resulting set is an MSTD set. We can almost add a random set in the middle; the reason we do not obtain a positive percentage is that we have the restriction that there can be no consecutive block of size $k$ of numbers in the middle that are not chosen to be in $A(M;k)$. This is easily satisfied by requiring us to choose at least one number in consecutive blocks of size $k/2$, and this is what leads to the loss of a positive percentage\footnote{Without this requirement, we could take any $M$ and thus would have a positive percentage work, specifically at least $2^{-(2k+2n)}$.} (though we do obtain sets that are known to be MSTD sets, and not just highly likely to be MSTD sets).

The paper is organized as follows. We describe our construction in \S\ref{sec:infinitefamilies}, and prove our claimed lower bounds for the percentage of sets that are MSTD sets in \S\ref{sec:lowerboundspercentage}. We then generalize our construction in \S\ref{sec:generalizingconstr} and explore when there are infinite families of sets satisfying \be \left|\gep_1 A + \cdots + \gep_n A\right| \ > \ \left|\widetilde{\gep}_1 A + \cdots + \widetilde{\gep}_n A\right|, \ \ \ \gep_i, \widetilde{\gep}_i \in \{-1,1\}. \ee We end with some concluding remarks and suggestions for future research in \S\ref{sec:concremfutureresearch}.


\section{Construction of infinite families of MSTD sets}\label{sec:infinitefamilies}

Let $A\subset [1,2n]$. We can write this set as $A=L\cup R$ where $L\subset[1,n]$ and $R\subset[n+1,2n]$. We have
\begin{equation}
A+A \ = \ [L+L] \cup [L+R] \cup [R+R]
\end{equation} where $L+L\subset[2,2n]$, $L+R \subset[n+2,3n]$ and $R+R\subset [2n+2,4n]$, and
\begin{equation}
A-A \ = \ [L-R]\cup [L-L] \cup [R-R] \cup [R-L]
\end{equation}
where $L-R\subset[-1,-2n+1]$, $L-L \subset[-(n-1),n-1]$, $R-R\subset [-(n-1),n-1]$ and $R-L\subset [1,2n-1]$.

A typical subset $A$ of $\{1,\dots,2n\}$ (chosen from the uniform model, see Footnote \ref{footnote:unifmodel}) will be a $P_n$-set (see Footnote \ref{footnote:beingpn}). It is thus the interaction of the ``fringe'' elements that largely determines whether a given set is an MSTD set. Our construction begins with a set $A$ that is both an MSTD set and a $P_n$-set. We construct a family of $P_n$, MSTD sets by inserting elements into the middle in such a way that the new set is a $P_n$-set, and the number of added sums is equal to the number of added differences. Thus the new set is also an MSTD set.

In creating MSTD sets, it is very useful to know that we have a $P_n$-set. The reason is that we have all but the ``fringe'' possible sums and differences, and are thus reduced to studying the extreme sums and differences. The following lemma shows that if $A$ is a $P_n$, MSTD set and a certain extension of $A$ is a $P_n$-set, then this extension is also an MSTD set. The difficult step in our construction is determining a large class of extensions which lead to $P_n$-sets; we will do this in Lemma \ref{lem:mainconstr}.

\begin{lem}\label{lem:stability}
  Let  $A=L\cup R$ be a $P_n$-set where $L\subset[1,n]$ and $R\subset[n+1,2n]$. Form $A'=L \cup M \cup R'$ where $M\subset [n+1,n+m]$ and $R'=R+m$. If $A'$ is a $P_n$-set then $|A'+A'|-|A+A|=|A'-A'|-|A-A|=2m$ (i.e., the number of added sums is equal to the number of added differences). In particular, if $A$ is an MSTD set then so is $A'$.
\end{lem}

\begin{proof}
We first count the number of added sums. In the interval $[2,n+1]$ both $A+A$ and $A'+A'$ are identical, as any sum can come only from terms in $L+L$. Similarly, we can pair the sums of $A+A$ in the region $[3n+1,4n]$ with the sums of $A'+A'$ in the region $[3n+2m+1,4n+2m]$, as these can come only from $R+R$ and $(R+m)+(R+m)$ respectively. Since we have accounted for the $n$ smallest and largest terms in both $A+A$ and $A'+A'$, and as both are $P_n$-sets, the number of added sums is just $(3n+2m+1)-(3n+1)=2m$.

Similarly, differences in the interval $[1-2n, -n]$ that come from $L-R$ can be paired with the corresponding terms from $L-(R+m)$, and differences in the interval $[n, 2n-1]$ from $R-L$ can be paired with differences coming from $(R+m)-L$. Thus the size of the middle grows from the interval $[-n+1,n-1]$ to the interval $[-n-m+1,n+m-1]$. Thus we have added $(2n+2m+3)-(2n+3)=2m$ differences. Thus $|A'+A'|-|A+A|=|A'-A'|-|A-A|=2m$ as desired.
\end{proof}

The above lemma is not surprising, as in it we assume $A'$ is a $P_n$-set; the difficulty in our construction is showing that our new set $A(M;k)$ is also a $P_n$-set for suitably chosen $M$. This requirement forces us to introduce the sets $O_i$ (which are blocks of $k$ consecutive integers), as well as requiring $M$ to have at least one of every $k$ consecutive integers.

We are now ready to prove the first part of Theorem \ref{thm:mainconstruction} by constructing an infinite family of distinct $P_n$, MSTD sets. We take a $P_n$, MSTD set and insert a set in such a way that it remains a $P_n$-set; thus by Lemma \ref{lem:stability} we see that this new set is an MSTD set.

\begin{lem}\label{lem:mainconstr}
  Let  $A=L\cup R$ be a $P_n$-set where $L\subset[1,n]$, $R\subset[n+1,2n]$, and $1,2n\in A$.
  Fix a $k \ge n$ and let $m$ be arbitrary. Choose any $M\subset[n+k+1,n+k+m]$ with the property that $M$ does not have a run of more than $k$ missing elements, and form $A(M;k)=L \cup O_1 \cup M \cup O_2 \cup R'$ where $O_1=[n+1,n+k]$, $O_2=[n+k+m+1,n+2k+m]$, and $R'=R+2k+m$. Then $A(M;k)$ is a $P_n$-set.
\end{lem}

\begin{proof} For notational convenience, denote $A(M;k)$ by $A'$.
Note $A'+A' \subset [2, 4n+4k+2m]$. We begin by showing that there are no missing sums from $n+2$ to $3n+4k+2m$; proving an analogous statement for $A'-A'$ shows $A'$ is a $P_n$-set. By symmetry\footnote{Apply the arguments below to the set $2n+2k+m-A'$, noting that $1, 2n+2k+m \in A'$.} we only have to show that there are no missing sums in $[n+2, 2n+2k+m]$. We consider various ranges in turn.

We observe that $[n+2, n+k+1]\subset A'+A'$ because we have $1\in L$ and these sums result from $1+O_1$. Additionally, $O_1 + O_1=[2n+2, 2n+2k]\subset A'+A'$. Since $n\le k$ we have $n+k+1\ge 2n+1$, these two regions are contiguous and thus $[n+2, 2n+2k]\subset A'+A'$.

Now consider $O_1 + M$. Since $M$ does not have a run of more than $k$ missing elements, the worst case scenario (in terms of getting the required sums) is that the smallest element of $M$ is $n+2k$ and that the largest element is $n+m+1$ (and, of course, we still have at least one out of every $k$ consecutive integers is in $M$). If this is the case then we still have $O_1+M \supset [(n+1)+(n+2k), (n+k) + (n+m+1)]=[2n+2k+1, 2n+k+m+1]$. We had already shown that $A'+A'$ has all sums up to $2n+2k$; this extends the sumset to all sums up to $2n+k+m+1$.

All that remains is to show we have all sums in $[2n+k+m+2, 2n+2k+m]$. This follows immediately from $O_1+O_2=[2n+k+m+2,2n+3k+m]\subset A'+A'$. This extends our sumset to include all sums up to $2n+3k+m$, which is well past our halfway mark of $2n+2k+m$. Thus we have shown that $A'+A' \supset [n+2, 3n+4k+2m+1]$.\\

We now do a similar calculation for the difference set, which is contained in $[-(2n+2k+m)+1, (2n+2k+m)-1]$. As we have already analyzed the sumset, all that remains to prove $A$ is a $P_n$-set is to show that $A'-A' \supset [-n-2k-m+1,n+2k+m-1]$. As all difference sets\footnote{Unless, of course, $A$ is the empty set!} are symmetric about and contain $0$, it suffices to show the positive elements are present, i.e., that $A'-A' \supset [1,n+2k+m-1]$.

We easily see $[1,k-1] \subset A'-A'$ as $[0,k-1] \subset O_1 - O_1$. Now consider $M -  O_1$. Again the worst case scenario (for getting the required differences) is that the least element of $M$ is $n+2k$ and the greatest is $n+m+1$. With this in mind we see that $M - O_1 \supset [(n+2k)-(n+k) , (n+m+1)-(n+1)]=[k,m]$. Now $O_2 - O_1 \supset [(n+k+m+1)-(n+k), (n+2k+m)-(n+1)]=[m+1,2k+m-1]$, and we therefore have all differences up to $2k+m-1$.

Since $2n \in A$ we have $2n+2k+m \in A'$. Consider $(2n+2k+m) - O_1 = [n+k+m,n+2k+m-1]$. Since $k\ge n$ we see that $n+k+m \le 2k+m$; this implies that we have all differences up to $n+2k+m-1$ (this is because we already have all differences up to $2k+m-1$, and $n+k+m$ is either less than $2k+m-1$, or at most one larger).
\end{proof}

\begin{proof}{Proof of Theorem \ref{thm:mainconstruction}(1).}
The proof of the first part of Theorem \ref{thm:mainconstruction} follows immediately. By Lemma \ref{lem:mainconstr} our new sets $A(M;k)$ are $P_n$-sets, and by Lemma \ref{lem:stability} they are also MSTD. All that remains is to show that the sets are distinct; this is done by requiring $n+k+1$ is not in our set (for a fixed $k$, these sets have elements $n+1, \dots, n+k$ but not $n+k+1$; thus different $k$ yield distinct sets). \end{proof}


\section{Lower bounds for the percentage of MSTDs}\label{sec:lowerboundspercentage}

To finish the proof of Theorem \ref{thm:mainconstruction}, for a fixed $n$ we need to count how many sets $\widetilde{M}$ of the form $O_1 \cup M \cup O_2$ (see Theorem \ref{thm:mainconstruction} for a description of these sets) of width $r = 2k+m$ can be inserted into a $P_n$, MSTD set $A$ of width $2n$. As $O_1$ and $O_2$ are just intervals of $k$ consecutive ones, the flexibility in choosing them comes solely from the freedom to choose their length $k$ (so long as $k \ge n$). There is far more freedom to choose $M$.

There are two issues we must address. First, we must determine how many ways there are there to fill the elements of $M$ such that there are no runs of $k$ missing elements. Second, we must show that the sets generated by this method are distinct. We saw in the proof of Theorem \ref{thm:mainconstruction}(1) that the latter is easily handled by giving $A(M;k)$ (through our choice of $M$) slightly more structure. Assume that the element $n+k+1$ is \emph{not} in $M$ (and thus not in $A$). Then for a fixed width $r=2k+m$ each value of $k$ gives rise to necessarily distinct sets, since the set contains $[n+1, n+k]$ but not $n+k+1$. In our arguments below, we assume our initial $P_n$, MSTD set $A$ is fixed; we could easily increase the number of generated MSTD sets by varying $A$ over certain MSTD sets of size $2n$. We choose not to do this as $n$ is fixed, and thus varying over such $A$ will only change the percentages by a constant independent of $k$ and $m$.

Fix $n$ and let $r$ tend to infinity. We count how many $\widetilde{M}$'s there are of width $r$ such that in $M$ there is at least one element chosen in any consecutive block of $k$ integers. One way to ensure this is to divide $M$ into consecutive, non-overlapping blocks of size $k/2$, and choose at least one element in each block. There are $2^{k/2}$ subsets of a block of size $k/2$, and all but one have at least one element. Thus there are $2^{k/2} - 1 = 2^{k/2} (1 - 2^{-k/2})$ valid choices for each block of size $k/2$. As the width of $M$ is $r-2k$, there are $\lceil \frac{r-2k}{k/2}\rceil \le \frac{r}{k/2}-3$ blocks (the last block may have length less than $k/2$, in which case any configuration will suffice to ensure there is not a consecutive string of $k$ omitted elements in $M$ because there will be at least one element chosen in the previous block). We see that the number of valid $M$'s of width $r-2k$ is at least $2^{r-2k} \left(1 - 2^{-k/2}\right)^{\frac{r}{k/2}-3}$. As $O_1$ and $O_2$ are two sets of $k$ consecutive $1$'s, there is only one way to choose either.

We therefore see that, for a fixed $k$, of the $2^r = 2^{m+2k}$ possible subsets of $r$ consecutive integers, we have at least $2^{r-2k} \left(1 - 2^{-k/2}\right)^{\frac{r}{k/2}-3}$ are permissible to insert into $A$. To ensure that all of the sets are distinct, we require $n+k+1 \not\in M$; the effect of this is to eliminate one degree of freedom in choosing an element in the first block of $M$, and this will only change the proportionality constants in the percentage calculation (and \emph{not} the $r$ or $k$ dependencies). Thus if we vary $k$ from $n$ to $r/4$ (we could go a little higher, but once $k$ is as large as a constant times $r$ the number of generated sets of width $r$ is negligible) we have at least some fixed constant times $2^r \sum_{k=n}^{r/4} \frac{1}{2^{2k}} \left(1 - 2^{-k/2}\right)^{\frac{r}{k/2}-3}$ MSTD sets; equivalently, the percentage of sets $O_1 \cup M \cup O_2$ with $O_i$ of width $k \in \{n,\dots, r/4\}$ and $M$ of width $r-2k$ that we may add is at least this divided by $2^r$, or some universal constant times \be\label{eq:keycardsum} \sum_{k=n}^{r/4} \frac1{2^{2k}} \left(1 - \frac1{2^{k/2}}\right)^{\frac{r}{k/2}} \ee (as $k \ge n$ and $n$ is fixed, we may remove the $-3$ in the exponent by changing the universal constant).

We now determine the asymptotic behavior of this sum. More generally, we can consider sums of the form \be S(a,b,c;r) \ = \ \sum_{k=n}^{r/4} \frac1{2^{ak}} \left(1 - \frac1{2^{bk}}\right)^{r/ck}. \ee For our purposes we take $a=2$ and $b=c=1/2$; we consider this more general sum so that any improvements in our method can readily be translated into improvements in counting MSTD sets. While we know (from the work of Martin and O'Bryant \cite{MO}) that a positive percentage of such subsets are MSTD sets, our analysis of this sum yields slightly weaker results. The approach in \cite{MO} is probabilistic, obtained by fixing the fringes of our subsets to ensure certain sums and differences are in (or not in) the sum- and difference sets. While our approach also fixes the fringes, we have far more possible fringe choices than in \cite{MO} (though we do not exploit this). While we cannot prove a positive percentage of subsets are MSTD sets, our arguments are far more elementary.

The proof of Theorem \ref{thm:mainconstruction}(2) is clearly reduced to proving the following lemma, and then setting $a = 2$ and $b=c=1/2$.

\begin{lem}\label{lem:lowerupperbounds} Let \be S(a,b,c;r) \ = \ \sum_{k=n}^{r/4} \frac1{2^{ak}} \left(1 - \frac1{2^{bk}}\right)^{r/ck}. \ee Then for any $\epsilon > 0$ we have \be \frac{1}{r^{a/b}} \ \ll \ S(a,b,c;r) \ \ll \ \frac{(\log r)^{2a+\epsilon}}{r^{a/b}}. \ee \end{lem}

\begin{proof} We constantly use $(1 - 1/x)^x$ is an increasing function in $x$. We first prove the lower bound. For $k \ge (\log_2 r)/b$ and $r$ large, we have \be \left(1 - \frac1{2^{bk}}\right)^{r/ck} \ = \ \left(1 - \frac1{2^{bk}}\right)^{2^{bk} \frac{r}{ck 2^{bk}}} \ \ge \ \left(1 - \frac{1}{r}\right)^{r \cdot \frac{b}{c \log_2 r}} \ \ge \ \foh \ee (in fact, for $r$ large the last bound is almost exactly 1). Thus we trivially have \bea S(a,b,c;r) & \ \ge \ & \sum_{k = (\log_2 r)/b}^{r/4} \frac1{2^{ak}} \cdot \foh \ \gg \ \frac1{r^{a/b}}. \eea

For the upper bound, we divide the $k$-sum into two ranges: (1) $bn \le bk \le \log_2 r - \log_2 (\log r)^\delta$; (2) $\log_2 r - \log_2(\log r)^\delta \le bk \le br/4$. In the first range, we have \begin{eqnarray} \left(1 - \frac1{2^{bk}}\right)^{r/ck} & \ \le \ & \left(1 - \frac{(\log r)^\delta}{r}\right)^{r/ck}\nonumber\\ & \ \ll \ & \exp\left(-\frac{b(\log r)^\delta}{c \log_2 r}\right) \nonumber\\ & \ \le \ & \exp\left(-\frac{b\log 2}{c} \cdot (\log r)^{\delta -1}\right). \end{eqnarray} If $\delta > 2$ then this factor is dominated by $r^{-\frac{b\log 2}{c} \cdot (\log r)^{\delta - 2}} \ll r^{-A}$ for any $A$ for $r$ sufficiently large. Thus there is negligible contribution from $k$ in range (1) if we take $\delta = 2 + \epsilon/a$ for any $\epsilon > 0$.

For $k$ in the second range, we trivially bound the factors $\left(1 - 1/2^{bk}\right)^{r/ck}$ by 1. We are left with \bea \sum_{k \ge \frac{\log_2 r}b - \frac{\log_2(\log r)^\delta}{b}} \ \frac1{2^{ak}} \cdot 1 \ \le \ \frac{(\log r)^{a \delta}}{r^{a/b}} \sum_{\ell = 0}^\infty \frac1{2^{a\ell}} \ \ll \ \frac{(\log r)^{a \delta}}{r^{a/b}}. \eea Combining the bounds for the two ranges with $\delta = 2 + \epsilon/a$ completes the proof.
\end{proof}

\begin{rek} The upper and lower bounds in Lemma \ref{lem:lowerupperbounds} are quite close, differing by a few powers of $\log r$. The true value will be at least $\left(\frac{\log r}{r}\right)^{a/b}$; we sketch the proof in Appendix \ref{sec:sizeSabcm}. \end{rek}

\begin{rek} We could attempt to increase our lower bound for the percentage of subsets that are MSTD sets by summing $r$ from $R_0$ to $R$ (as we have fixed $r$ above, we are only counting MSTD sets of width $2n+r$ where $1$ and $2n+r$ are in the set. Unfortunately, at best we can change the universal constant; our bound will still be of the order $1/R^4$. To see this, note the number of such MSTD sets is at least a constant times $\sum_{r=R_0}^R 2^r / r^4$ (to get the percentage, we divide this by $2^R$). If $r \le R/2$ then there are exponentially few sets. If $r \ge R/2$ then $r^{-4} \in [1/R^4, 16/R^4]$. Thus the percentage of such subsets is still only at least of order $1/R^4$.
\end{rek}


\section{Generalizing our construction}\label{sec:generalizingconstr}

Instead of searching for $A$ such that $|A+A| > |A-A|$, we now consider the more general problem\footnote{We do not consider the most general problem of comparing arbitrary combinations of $A$, contenting ourselves to this special case; see \cite{HM} for some thoughts about such generalizations.} of when \be \left|\gep_1 A + \cdots + \gep_n A\right| \ > \ \left|\widetilde{\gep}_1 A + \cdots + \widetilde{\gep}_n A\right|, \ \ \ \gep_i, \widetilde{\gep}_i \in \{-1,1\}. \ee

Consider the generalized sumset \be f_{j_1,\ j_2}(A)\ =\ A+A+\cdots+A-A-A-\cdots-A,\ee where there are $j_1$ pluses\footnote{By a slight abuse of notation, we say there are two sums in $A+A-A$, as is clear when we write it as $\epsilon_1 A + \epsilon_2 A + \epsilon_3 A$.} and $j_2$ minuses, and set $j=j_1 + j_2$. Our notion of a $P_n$-set generalizes, and we find that if there exists one set $A$ with $|f_{j_1,\ j_2}(A)| > |f_{j_1',\ j_2'}(A)|$, then we can construct infinitely many such $A$.
Note without loss of generality that we may assume $j_1 \ge j_2$.\footnote{This follows as we are only interested in $|f_{j_1,\ j_2}(A)|$, which equals $|f_{j_2,\ j_1}(A)|$. This is because $B$ and $-B$ have the same cardinality, and thus (for example) we see $A+A-A$ and $-(A-A-A)$ have the same cardinality.}

\begin{defi}[$P_n^j$-set.] Let $A \subset [1, k]$ with $1, k, \in A$. We say $A$ is a $P^j_n$-set if any $f_{j_1,\ j_2}(A)$ contains all but the first $n$ and last $n$ possible elements. \end{defi}

\begin{rek} Note that a $P_n^2$-set is the same as what we called a $P_n$-set earlier. \end{rek}

We expect the following generalization of Theorem \ref{thm:mainconstruction} to hold.

\begin{conj}\label{conj:genconstr} For any $f_{j_1,\ j_2}$ and $f_{j_1',\ j_2'}$, if there exists a finite set of integers $A$ which is (1) a $P^j_n$-set; (2) $A \subset [1, 2n]$ and $1, 2n \in A$; and (3) $|f_{j_1,\ j_2}(A)|>|f_{j_1',\ j_2'}(A)|$, then there exists an infinite family of such sets. \end{conj}

The difficulty in proving the above conjecture is that we need to find a set $A$ satisfying $|f_{j_1,\ j_2}(A)|>|f_{j_1',\ j_2'}(A)|$; once we find such a set, we can mirror the construction from Theorem \ref{thm:mainconstruction}. Currently we can only find such $A$ for $j \in \{2,3\}$:

\begin{thm}\label{thm:genconstr} Conjecture \ref{conj:genconstr} is true for $j \in \{2,3\}$. \end{thm}

As the proof is similar to that of Theorem \ref{thm:mainconstruction}, we just highlight the changes. We prove the lemmas below in greater generality than we need for our theorem as this generality is needed to attack Conjecture \ref{conj:genconstr}. The first step is an analogue of Lemma \ref{lem:stability}, the second is proving that a $P_n^2$-set is also a $P_n^j$-set, and the third is constructing sets $A$ (when $j = 3$) to start the construction.

\begin{lem}\label{lem:gen1} Let $A=L \cup R$ be a $P_n^j$-set, where $L \subset [1, n], R \subset [n+1, 2n]$. Form $A' =L\cup M\cup R',$ where $M \subset [n+1, n+m]$ and $R' = R+m$.  If $A'$ is a $P_n^j$-set, then $|f_{j_1,\ j_2}(A')| -|f_{j_1,\ j_2}(A)| = |f_{j_1',\ j_2'}(A')| -|f_{j_1',\ j_2'}(A)|.$  Thus if $|f_{j_1,\ j_2}(A)|>|f_{j_1',\ j_2'}(A)|$, the same is true for $A'$. \end{lem}

\begin{proof} Since $A\subset [1, 2n]$ and is a $P^j_n$-set, we know $f(A) \subset [j_1-2nj_2, 2nj_1-j_2]$ and $[j_1 -2nj_2+n, 2nj_1-j_2 -n] \subset f(A)$. Note any elements in $f(A) \cap [j_1 -2nj_2, j_1-2nj_2+n-1]$ can only come from $L+L+L+ \cdots+L-R-R-R-\cdots-R$.

As $A' \subset [1, 2n+m]$, $f(A') \subset [j_1-(2n+m)j_2 , (2n+m)j_1-j_2]$ and $[j_1 -(2n+m)j_2+n, (2n+m)j_1-j_2 -n] \subset f(A)]$. Any elements in $f(A) \cap [j_1 -(2n+m)j_2, j_1-(2n+m)j_2+n-1]$ can only come only from $L+L+L+ \cdots+L-R'-R'-R'-\cdots-R'$, which is simply a translation of $L+L+L+ \cdots+L-R-R-R-\cdots-R$.

A similar argument works for the right fringe of $f_{j_1,\ j_2}(A')$. Thus $|f(A')| = |f(A)| +jm$ (this is because the potential width of $f_{j_1,\ j_2}(A')$ is $jm$ more than that of $f_{j_1,\ j_2}(A)$, and the two fringes of these sets are in a 1-1 correspondence). Since $|f_{j_1,\ j_2}(A')| - |f_{j_1,\ j_2}(A)|$ depends only on $j=j_1+j_2$, it holds for any pair of forms with $j$ coefficients, and the lemma is proven. \end{proof}


\begin{lem}\label{lem:gen2} For $j \ge 3$, any $P_n^2$-set is also a $P^j_n$-set. \end{lem}

\begin{proof} Let A be a $P_n^2$-set, where $A \subset [1, k]$ and
$1, k \in A$.  Assume $k \geq 2n$. Then $A+A \cap [n+2, 2k-n] = [n+2, 2k-n]$ (as $A$ is a $P_n^2$-set).

Let $f_{j_1,\ j_2}$ be a form with $j \ge 3$, and thus either $j_1$ or $j_2$ is at least 2; without loss of generality we assume $j_1 \ge 2$. There is a form $f_{j_1-2,\ j_2}$ such that $f_{j_1-2,\ j_2}(A) +A+A = f_{j_1,\ j_2}(A)$.  The proof follows by showing $f_{j_1-2,\ j_2}(\{1, k\})+A+A$ contains all necessary elements, namely $[j_1-kj_2 +n, j_1k - j_2 -n]$. (By $f_{j_1-2,\ j_2}(\{1, k\})$ we mean all numbers of the form $\gep_1 a_1 + \cdots + \gep_{j-2} a_{j-2}$, with the $\gep_i$ the coefficients of the form $f_{j_1-2,j_2}$ and $a_i \in \{1,k\}$.) We have \begin{equation} f_{j_1-2,\ j_2}(\{1, k\})\ \supset \ \{j_1 - 2 -i +k(i-j_2) \;|\; 0\leq i \leq j -2\}. \end{equation} To see this, we first consider $i \le j_1-2$. For such $i$, for the positive summands choose $1$ a total of $j_1-2-i$ times and $k$ a total of $i$ times, while for the negative summands we choose $k$ each of the $j_2$ times. If now $j_1 - 2 < i \le j-2$, for the positive summands we choose $k$ a total of $i-j_2$ times (which is permissible as this is at most $j_1-2$) and we choose 1 the remaining $j_1-2-(i-j_2)$ times, while for the negative summands we choose $1$ all $j_2$ times. This leads to a sum of $k\cdot (i-j_2) + 1\cdot(j_1-2+j_2-i)-1\cdot j_2$, which equals $j_1-2-i+ k(i-j_2)$ as claimed. Unfortunately, this argument fails if $i=j_1-1$ and $j_1=j_2$, as we would then be choosing $k$ from the positive summands negative one times.\footnote{This is the only bad case we need consider, as we know $j_1 \ge j_2$, and the only problem arises when $i-j_2 < 0$.} We are thus left with showing that we may obtain the sum $-1-k$ in this special case. As $j_1=j_2$, we just choose $1$ for the $j_1-2$ positive summands and $-1$ for all but one of the $j_2$ negative summands (where we choose one to be $k$).

As $A$ is a $P_n^2$-set, $A+A \supset [n+2, 2k-n]$. Thus \begin{eqnarray} \bigcup_{i=0}^{j-2} [L_i, U_i] & \ \subset \ &
f_{j_1-2,\ j_2}(\{1, k\})+A+A, \end{eqnarray}
where \begin{eqnarray} L_i & \  = \ &  j_1 - 2 -i +k(i-j_2)+ n+2 \nonumber\\ U_i &=& j_1 - 2 -i +k(i-j_2) +2k-n. \end{eqnarray}

We see that $L_0 = j_1-kj_2 +n$ and $U_{j-2} = j_1k - j_2 -n$, our two desired endpoints. The proof is completed by showing the intervals $[L_i,U_i]$ cover the desired interval and has no gap with its neighbors.

Since $2n \leq k$, we have:
\begin{eqnarray} L_i -1 & \ = \ &  j_1  -i +k(i-j_2)+ n -1\nonumber\\ &=& (j_1-i+ki-j_2k -1) +n\nonumber\\ & \le &
(j_1-i+ki-j_2k -1) +k- n \nonumber\\
&=& j_1 - 2 -(i-1) +k((i-1)-j_2) +2k-n \nonumber\\ &\le & U_{i-1}. \end{eqnarray}
Thus there are no gaps between the intervals $[L_{i-1}, U_{i-1}], \; [L_i, U_i]$ and they therefore cover the necessary range. \end{proof}

\begin{rek} Note that the above lemma is false if the size of $n$ is unrestricted. To take an extreme example, let $A = \{1, 10\}$ and $n= 9$. Then $A$ is a $P_n^2$-set ($11 \in A+A, \; 0 \in A-A$) but $A$ is not a $P^3_n$-set. \end{rek}

\begin{proof}[Proof of Theorem \ref{thm:genconstr}] Lemmas \ref{lem:gen1} and \ref{lem:gen2} imply that the sets described in Lemma \ref{lem:mainconstr} also work in our generalized case.  The counting argument of \S\ref{sec:lowerboundspercentage} requires no modification. Thus the theorem is proved \emph{provided} we can find an $A$ to start the process.

The following set was obtained by taking elements in $\{2,\dots,49\}$ to be in $A$ with probability\footnote{Note the probability is 1/3 and not 1/2.} $1/3$ (and, of course, requiring $1, 50 \in A$); it took about 300000 sets to find the first one satisfying our conditions: \be A \ = \ \{1, 2, 5, 6, 16, 19, 22, 26, 32, 34, 35, 39, 43, 48, 49,  50\}. \ee To be a $P_{25}^3$-set we need to have $A+A+A \supset [n + 3, 6 n - n] = [28, 125]$ and $A+A-A \supset [-n + 2, 3 n - 1] = [-23, 74]$. A simple calculation shows $A+A+A = [3,150]$, all possible elements, while $A+A-A = [-48, 99] \backslash \{-34\}$ (i.e., every possible element but -34). Thus $A$ is a $P_{25}^3$-set satisfying $|A+A+A| > |A+A-A|$, and thus we have the example we need to prove Theorem \ref{thm:genconstr}.
\end{proof}

\begin{rek} We could also have taken \be A \ = \ \{1, 2, 3, 4, 8, 12, 18, 22, 23, 25, 26, 29, 30, 31, 32, 34,    45, 46, 49, 50\}, \ee which has the same $A+A+A$ and $A+A-A$. \end{rek}


\section{Concluding remarks and future research}\label{sec:concremfutureresearch}

One avenue of future research is to complete the proof of Conjecture \ref{conj:genconstr} and give an elementary example of an infinite family of sets satisfying $|f_{j_1,\ j_2}(A)| > |f_{j_1',\ j_2'}(A)|$. We have reason to believe the correct model is to look for $P_n^j$-sets by choosing the numbers $\{2,\dots,2n-1\}$ to be in $A$ with probability $1/j$ (and, of course, requiring $1, 2n \in A$). Unfortunately the density of such sets appears to decrease rapidly with $n$, and to date straightforward computer searches have been unsuccessful when $j=4$. As we shall see below, perhaps a better algorithm would incorporate choosing elements near the fringes (i.e., near $1$ and $2n$) with a different probability than $1/j$. \\

We also observed earlier (Footnote \ref{footnote:beingpn}) that for a constant $0<\alpha \le 1$, a set randomly chosen from $[1,2n]$ is a $P_{\lfloor \alpha n \rfloor}$-set with probability approaching $1$ as $n\to\infty$. MSTD sets are of course not random, but it seems logical to suppose that this pattern continues.

\begin{conj}\label{conj:MSTDsP_n}
Fix a constant $0<\alpha\le 1/2$. Then as $n\to\infty$ the probability that a randomly chosen MSTD set in $[1,2n]$ containing $1$ and $2n$ is a $P_{\lfloor \alpha n \rfloor}$-set goes to $1$.
\end{conj}

In our construction and that of \cite{MO}, a collection of MSTD sets is formed by fixing the fringe elements and letting the middle vary. The intuition behind both is that the fringe elements matter most and the middle elements least. Motivated by this it is interesting to look at all MSTD sets in $[1,n]$ and ask with what frequency a given element is in these sets. That is, what is \be \gamma(k;n) \ = \ \frac{\#\{A: k \in A\ {\rm and}\ A\ {\rm is\ an\ MSTD\ set}\}}{\#\{A: \ A\ {\rm is\ an\ MSTD\ set}\}} \ee as $n\to\infty$?
We can get a sense of what these probabilities might be from Figure \ref{fig:MSTDfreq}.

\begin{figure}
\begin{center}
\scalebox{1.75}{\includegraphics{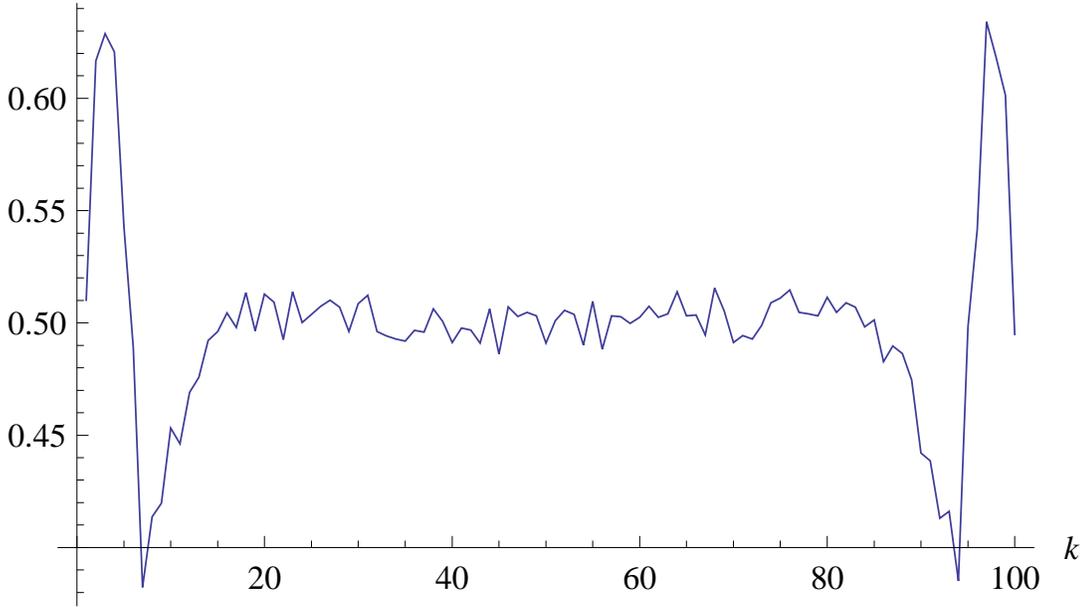}}
\caption{\label{fig:MSTDfreq}\textbf{Estimation of $\gamma(k,100)$ as $k$ varies from $1$ to $100$ from a random sample of 4458 MSTD sets.}}
\end{center}\end{figure}

Note that, as the graph suggests, $\gamma$ is symmetric about $\frac{n+1}2$, i.e. $\gamma(k,n)=\gamma(n+1-k,n)$. This follows from the fact that the cardinalities of the sumset and difference set are unaffected by sending $x \to \alpha x + \beta$ for any $\alpha, \beta$. Thus for each MSTD set $A$ we get a distinct MSTD set $n+1-A$ showing that our function $\gamma$ is symmetric. These sets are distinct since if $A=n+1-A$ then $A$ is sum-difference balanced.\footnote{The following proof is standard (see, for instance, \cite{Na2}). If $A = n+1-A$ then \be |A+A| \ = \ |A + (n+1-A)| \ = \ |n+1+(A-A)| \ = \ |A-A|.\ee}

From \cite{MO} we know that a positive percentage of sets are MSTD sets. By the central limit theorem we then get that the average size of an MSTD set chosen from $[1,n]$ is about $n/2$. This tells us that on average $\gamma(k,n)$ is about $1/2$. The graph above suggests that the frequency goes to $1/2$ in the center. This leads us to the following conjecture:

\begin{conj}\label{conj:50conjecture}
Fix a constant $0<\alpha<1/2$. Then $\lim_{n\rightarrow\infty}{\gamma(k,n)}=1/2$ for $\lfloor \alpha n \rfloor \le k \le n - \lfloor \alpha n \rfloor$.
\end{conj}

\begin{rek} More generally, we could ask which non-decreasing functions $f(n)$ have $f(n)\rightarrow \infty$, $n-f(n)\rightarrow \infty$ and $\lim_{n\rightarrow\infty}{\gamma(k,n)}=1/2$ for all $k$ such that $\lfloor f(n) \rfloor \le k \le n - \lfloor f(n) \rfloor$.
\end{rek}

\appendix



\section{Size of $S(a,b,c;r)$}\label{sec:sizeSabcm}

We sketch the proof that the sum \be S(a,b,c;r) \ = \ \sum_{k=n}^{r/4} \frac1{2^{ak}} \left(1 - \frac1{2^{bk}}\right)^{r/ck} \ee is at least $\left(\frac{\log r}{r}\right)^{a/b}$. We determine the maximum value of the summands \be f(a,b,c;k,r) \ = \ \frac1{2^{ak}} \left(1 - \frac1{2^{bk}}\right)^{r/ck}. \ee Clearly $f(a,b,c;k,r)$ is very small if $k$ is small due to the second factor; similarly it is small if $k$ is large because of the first factor. Thus the maximum value of $f(a,b,c;k,r)$ will arise not from an endpoint but from a critical point.

It is convenient to change variables to simplify the differentiation. Let $u = 2^k$ (so $k = \log u / \log 2$). Then \be g(a,b,c;u,r) \ = \ f(a,b,c;k,r) \ = \ u^{-a} \left(1 - \frac1{u^b}\right)^{u^b \cdot\frac{m\log 2}{cu^b\log u}}. \ee Thus \be g(a,b,c;u,r) \ \approx \ u^{-a} \exp\left(-\frac{r\log 2}{cu^b\log u}\right).\ee Maximizing this is the same as minimizing $h(a,b,c;u,r) = 1/g(a,b,c;u,r)$. After some algebra we find \bea h'(a,b,c;u,r) & \ = \ & \frac{h(a,b,c;u,r)}{cu \log^2 u} \left(ac u^b \log^2 u -r \log 2 \cdot \left(b \log u + 1\right) \right).\ \ \   \eea Setting the derivative equal to zero yields \bea ac u^b \log^2 u & \ =\ & r \log 2 \cdot \left(b \log u + 1\right). \eea As we know $u$ must be large, looking at just the main term from the right hand side yields \bea ac u^b \log u \ \approx \ r b \log 2, \eea or \be\label{eq:utotheblogu} u^b \log u \ \approx \ C r, \ \ \ C \ = \ \frac{b \log 2}{ac}. \ee To first order, we see the solution is \be u_{\max} \ = \ \left(\frac{(Cr)}{\frac{\log(Cr)}{b}}\right)^{\frac1{b}} \ \approx \ C' \left(\frac{r}{\log r}\right)^\frac{1}{b}. \ee Straightforward algebra shows that the maximum value of our summands is approximately $(C'e^{1/b})^{-a} \left(\frac{\log r}{r}\right)^{a/b}$.


\bigskip

\ \\

\end{document}